\documentclass[a4paper]{article}

\usepackage{amsmath,amsthm,amssymb}
\usepackage{times}
\usepackage{enumerate}
\usepackage{graphicx}
\usepackage{dsfont}
\usepackage{subfig}
\usepackage[normalem]{ulem}
\usepackage{cancel}
\RequirePackage[numbers]{natbib}
\RequirePackage[colorlinks,citecolor=blue,linkcolor=blue,urlcolor=blue]{hyperref}

\def\titlerunning#1{\gdef\titrun{#1}}
\makeatletter
\def\author#1{\gdef\autrun{\def\and{\unskip, }#1}\gdef\@author{#1}}
\def\address#1{{\def\and{\\\hspace*{18pt}}\renewcommand{\thefootnote}{}%
  \footnote {#1}}%
\markboth{\autrun}{\titrun}}
\makeatother

\def\subjclass#1{{\renewcommand{\thefootnote}{}%
\footnote{\emph{Mathematics Subject Classification (2020):} #1}}}

\newtheorem{theorem}{Theorem}[section]

\newtheorem{lemma}[theorem]{Lemma}
\newtheorem{proposition}[theorem]{Proposition}

\theoremstyle{definition}

\newtheorem{remark}[theorem]{Remark}

\DeclareMathOperator{\Var}{Var}

\numberwithin{equation}{section}

\allowdisplaybreaks

\begin{document}

\newcommand{\R}{\mathbb{R}}
\newcommand{\I}{\mathbb{I}}
\newcommand{\Q}{\mathbb{Q}}
\newcommand{\Z}{\mathbb{Z}}
\newcommand{\N}{\mathbb{N}}
\newcommand{\eps}{\varepsilon}
\newcommand{\p}{\mathbb{P}}
\newcommand{\F}{\mathcal{F}}
\newcommand{\e}{\mathcal{E}}
\newcommand{\Oe}{\mathcal{E}^{OU}\!}
\newcommand{\E}{\mathbb{E}}
\newcommand{\mC}{\mathcal{C}}
\newcommand{\B}{\mathcal{B}}
\newcommand{\cdl}{c\`{a}dl\`{a}g }
\newcommand{\D}{\mathrm{D}}
\newcommand{\Li}{L_2^{\uparrow}}
\newcommand{\FC}{\mathcal{FC}}
\newcommand{\OXi}{\Xi^{OU}\!}
\newcommand{\OLambda}{\Lambda^{OU}}
\newcommand{\Deriv}{\mathop D}
\newcommand{\bDeriv}{{\mathop D}^\tau}
\renewcommand{\L}{\mathrm{L}}
\newcommand{\Cint}{K_{\Sigma,\Omega}}
\newcommand{\neumann}{\Sigma_{\operatorname{N}}}
\newcommand{\mus}{\lambda_\Sigma}
\newcommand{\mun}{\lambda_{\operatorname{N}}}
\newcommand{\DomE}{\mathcal {D}}
\newcommand{\Ric}{\operatorname{Ric}}
\newcommand{\id}{\operatorname{id}}

\newcommand{\stkout}[1]{\ifmmode\text{\sout{\ensuremath{#1}}}\else\sout{#1}\fi}

\def\one{\mathbb I}
\def\Bdelta{\Delta^\tau}
\def\bnabla{\nabla^\tau}

\baselineskip=17pt

\title{Spectral gap estimates for Brownian motion on domains with sticky-reflecting boundary diffusion}

\titlerunning{Spectral gap for domains with boundary diffusion}

\author{Vitalii Konarovskyi$^{\dagger}$, Victor Marx$^*$, and Max von Renesse$^*$}

\date{\today}

\maketitle

\address{$\dagger$ University of Hamburg, Bundesstraße 55, 20146 Hamburg, Germany; $*$ Universit\"{a}t Leipzig, Fakult\"{a}t f\"{u}r Mathematik und Informatik, Augustusplatz 10, 04109 Leipzig, Germany\\
E-mail: vitalii.konarovskyi@uni-hamburg.de, marx@math.uni-leipzig.de, renesse@uni-leipzig.de }


\begin{abstract} Introducing  an interpolation method we derive lower bounds for the spectral gap for Brownian motion on general domains with sticky-reflecting boundary diffusion associated to the first nontrivial eigenvalue for the Laplace operator with corresponding Wentzell-type boundary condition. In the manifold case our proofs involve  novel applications of the celebrated Reilly formula. 

\end{abstract}

\section{Introduction and statement of main result}
Brownian motion on smooth domains with sticky-reflecting  diffusion along the boundary has a long history, dating back at least to  Wentzell \cite{MR121855}. The induced probability evolution interpolates between Neumann heat flow inside the domain and tangential heat flow along boundary as marginal cases. 

In concise form such a process is determined by the associated of Feller generator $(A,\mathcal D(A))$ on the Banach space  $C_0(\overline \Omega)$ of continuous functions vanishing at infinity on a  smooth manifold with boundary $ \overline \Omega= \Omega \cup \partial \Omega$, where
\begin{equation}
\label{eq:gendef}
    \begin{aligned} 
 A f &=  \Delta f \one_{\Omega} + \left( \beta  \Bdelta f - \gamma \frac {\partial f}{\partial \nu }\right)\one_{\partial \Omega},\\
 \mathcal D(A) &=  \{ f \in C_0(\overline \Omega) \, : \, Af \in C_0(\overline \Omega) \}
 \end{aligned}
 \end{equation}
Here  
$\frac{\partial}{\partial \nu}$ is the outer normal derivative, $\Bdelta$ is the Laplace-Beltrami operator on the boundary $\partial \Omega$. The parameters $\beta,\gamma \in[0,\infty)$ denote the diffusion resp.\ inward sticky reflection rate at the boundary. For $\gamma =0$ the trajectories of the process will eventually collapse to a conventional Brownian motion along the boundary. The limiting case $\gamma=\infty$ corresponds to instantaneous boundary reflection, i.e.\ macroscopically standard Neumann heat flow in the bulk. Physically, such models correspond to heat conduction in materials with special coating.

Motivated by stochastic applications,
our interest in this work is a lower estimate on the first non trivial eigenvalue 
of the operator $-A$. By construction, the eigenvalue $\sigma$ appears simultaneously in two coupled PDE for the interior and the boundary. More specifically, $\sigma \geq 0$ is an eigenvalue, if for some $f\in \mathcal D(A)$ 
\begin{equation} \label{eq:full_ev_pde}
    \begin{cases}
\Delta f + \sigma f = 0 & \mbox{ in } \Omega,\\
\beta \Bdelta f - \gamma \frac{\partial f}{\partial \nu} + \sigma f =0 & \mbox{ on } \partial \Omega.
\end{cases}
\end{equation}

To the best of our knowledge, the spectral problem  has not been studied previously in this form. In contrast to classical Robin, Steklov, or Wentzell-type eigenvalue problems, the spectral parameter enters \eqref{eq:full_ev_pde} simultaneously in the interior equation and explicitly in the boundary condition, leading to a genuinely coupled spectral structure.

The goal is to give lower bounds on $\sigma$,  using minimal geometric information on $\Omega$ and $\partial \Omega$ in form of lower curvature bounds  and upper bounds on the Poincar\'e constants $C_\Omega$ and $C_\partial$ of the bulk resp.\ the boundary. 

\bigskip 

Below we establish such estimates for $\beta=1$, the cases $\beta >0$ and $\beta =0$ can be treated in an analogous fashion. To this aim it is convenient to introduce the reparametrization of $\gamma$ in terms of $\alpha\in [0,1]$ via 
\begin{equation}\label{eq:def_of_alpha}
\gamma = \frac \alpha {1-\alpha} \frac {|\partial \Omega|}{|\Omega|}. 
\end{equation} 

Let $\Ric$ be the Ricci curvature of $\Omega$ and $\mathrm{II}$ be the second fundamental form on the boundary $\partial \Omega$. The main result of this work reads then as  follows. 
\begin{theorem}\label{thm_main} Let $\Omega$ be a smooth compact  $d$-dimensional  Riemmannian manifold with boundary $\partial \Omega$ such that for some $k_R, k_2 >0$  
\[\Ric|_{\Omega}\geq k_R \id \quad \text{and} \quad \mathrm{II}|_{\partial \Omega} \geq k_2 \id. 
\]
Then for $\sigma=\sigma_\alpha$ in \eqref{eq:full_ev_pde}
 with $\alpha \in [0,1]$ given by~\eqref{eq:def_of_alpha}   
it holds that
\[
\sigma_\alpha \geq 
\max \left\{\frac 1 {M_{1,\alpha}},\frac 1 {M_{2,\alpha}}\right\},
\]
where 
\[M_{1,\alpha}=
 \max \left\{C_\Omega + \frac{(1-\alpha)(d-1)}{d k_R} ,
\frac{C_\partial}{d k_R}\cdot \frac{2(1-\alpha) dk_R{|\partial \Omega|} +\alpha d k_2 k_R C_\Omega{|\Omega|} +\alpha(1-\alpha)(d-1)k_2{|\Omega|} }{2(1-\alpha){|\partial \Omega|}+\alpha  k_2 C_{\partial}{|\Omega|}} \right\}
\] 
and 
\[ M_{2,\alpha}  =   \max\left\{\frac{(3d-1)(1-\alpha) }{d\alpha k_2}\frac{|\Omega|}{|\partial \Omega |}, \frac{d-1}{d k_R}\right\}.\]
\end{theorem}

\begin{remark}
(i) 
From the proof below it follows that Theorem~\ref{thm_main} remains true when we use any upper bound for $C_\Omega$ and $C_\partial$ in the definition of $M_{1,\alpha}$. 
    
(ii) It is possible to recast the problem \eqref{eq:full_ev_pde} as an eigenvalue problem on $\partial \Omega$ involving the Dirichlet-to-Neumann map for the (singular) shifted operator $\Lambda_\sigma = \Delta + \sigma$. This establishes a similarity to the Steklov eigenvalue problem and its variants (cf. literature review below), but in this representation the problem becomes nonlinear in $\sigma$. We will not pursue this approach and work with energy arguments instead. However, the classical Steklov eigenvalue of $\Omega$ will appear as a key component of our interpolation approach below.

(iii) Our theorem above is formulated under strict positivity assumption for the Ricci curvature of the bulk and the second fundamental form of the boundary. As it is known from classical eigenvalue estimates, the passage to general lower bounds may change the situation significantly, and other arguments need to be used, cf. \cite{Bormann2026StickyReflecting,Bormann2026Functional,bormann2025cheegertypeinequalitydriftlaplacian,wang2025superweakpoincareinequalities}

(iv) The two lower bounds involving $M_{1,\alpha}$ and $M_{2,\alpha}$ are obtained in different ways. $M_2$ is obtained by  direct application of the Reilly formula to the eigenfunction associated to  $\sigma_\alpha$ and makes no use of $C_\Omega$ and $C_\partial$. However, the resulting estimate is poor in many cases. Therefore we use a second method based on the interpolating nature of the Process to relate $C_\alpha$ with $C_\Omega$ and $C_\partial$ to give the sharper bound $M_1$. The latter is what we call \textit{interpolation method} which can be applied in more general settings as will be shown be examples further below.
\end{remark}

The structure of the paper is as follows. In section \ref{sec:direct} we establish the basic notation and variational setup for the analysis of the problem. We derive the estimate $M_2$ via a direct application of the Reilly formula. Section \ref{sec:interp_method} contains the presentation of the basic interpolation method, which is used in section in in detail to establish the second bound $M_1$. A critical ingredient is the boundary-bulk Poincar\'e estimate  
\[
\int_\Omega f^ 2  dx \leq \frac {d-1}{d  k_R} \int_\Omega  |\nabla f|^2 dx, \]
for all  $f\in C^1(\Omega)$ with  $\int_{\partial \Omega} f dS =0$, which we establish with an independent appiccation of the Reilly formula. 
In section \ref{Sec:examples} we demonstrate that the interpolation method can also be used successfully in cases when $k_R=0$ through the example of Euclidean balls, including the situation of tangential diffusion along only parts of the boundary and plain reflection elsewhere. A final example discusses the problem in case of a  diffusion on a stratified space. We conclude in section \ref{subsection:cont_section}  with a discussion for sufficient conditions on the continuity of $C_\alpha$ at $\alpha \in \{0,1\}$, which may fail, in general.

\subsection*{Literature Review}

\textit{1. Constructions of diffusions with dynamic boundary conditions
(stochastic vs.\ analytic).}
Diffusions with boundary interaction and sticky reflection have a long history, dating back at least to Wentzell~\cite{MR121855}.
On the stochastic side, rigorous process constructions on special domains were obtained in
\cite{MR126883,MR929208,MR287612} and later extended to more general settings, including jump--diffusion processes cf.\ \cite{MR245085}; see also~\cite{MR4176673}.
Related models with $\beta>0$ appear in interacting particle systems with singular boundary behaviour or zero-range interaction, see
\cite{MR4096131,MR2198199,MR2215623,konarovskyi2017reversible,nonnenmacher2018overdamped}.
In symmetric situations, an efficient construction via Dirichlet forms was provided in~\cite{MR3498008,grothaus,Ullrich:2025}. 
Maximal regularity properties of the associated semigroups was studied in~\cite{MR4065110}. 
The large-deviation properties of sticky-reflecting diffusions and Wasserstein gradient-flow structure of  the associated Fokker--Planck evolutions are addressed in the recent works   \cite{CasterasMonsaingeonSantambrogio2024,CasterasMonsaingeonNenna2025,BormannMonsaingeonRengerVonRenesse2025}.

\smallskip \noindent \textit{2. Spectral problems with generalized boundary conditions 
(Robin, Steklov and Wentzell-type).}

For context,  we recall three other 'classical ' eigenvalue problems in spectral geometry. 
\begin{itemize}
    \item 
The Robin Laplacian eigenvalue problem 
\[
\begin{cases}
-\Delta u = \sigma u & \text{in } \Omega,\\
\partial_\nu u + \beta u = 0 & \text{on } \partial\Omega,
\end{cases}
\]
which has been studied intensively in the literatre, see \cite{BucurFreitasKennedy2017} for a review.

\item 
The Steklov eigenvalue problem,
\[
\begin{cases}
\Delta u = 0 & \text{in } \Omega,\\
\partial_\nu u = \sigma u & \text{on } \partial\Omega,
\end{cases}
\]
which provides the spectral problem for the Dirichlet-to-Neumann map; 
see, e.g., the survey \cite{MR3662010}
\item The 'Wentzell eigenvalue' problem 
\[
\begin{cases}
\Delta u = 0 & \text{in } \Omega,\\
-\beta \Delta_{\partial\Omega} u + \partial_\nu u = \sigma u 
& \text{on } \partial\Omega,
\end{cases}
\]
which was studied recently in e.g.\ \cite{DambrineKatebLamboley2016Wentzell} and \cite{XiaWang2018}.
Kennedy~\cite{Kennedy2010RobinWentzell}
studies a variant of this, where $\Delta_{\partial\Omega} u$ is replaced by the boundary trace of $\Delta_{\Omega} u$,  under the name of a ('generalized') Wentzell boundary condition.
\end{itemize}
The eigenvalue problem of this work \eqref{eq:full_ev_pde} differs structurally from all of the above, since $\sigma$ appears both in the interior equation and explicitly in the boundary condition, where it is coupled to the tangential Laplace–Beltrami operator. In fact, as suggested by the gradient-flow perspective on the associated dynamics \cite{CasterasMonsaingeonSantambrogio2024,BormannMonsaingeonRengerVonRenesse2025} problem \eqref{eq:full_ev_pde} has the structure of an eigenvalue problem on a closed manifold with degenerate geometry on codimension-one subsets. -- Methodologically, our analysis below is particularly inspired by the works \cite{Kennedy2010RobinWentzell,kolesnikov,MR3951758} on manifolds with boundary, where the Reilly formula plays a major role.

\section{Direct Approach} \label{sec:direct}
 \subsection{Symmetrizing measure and  variational formulation}

In what follows we consider the situation of a bounded smooth domain $\Omega$. In this case the operator $A$ acting on smooth functions $C^\infty(\overline \Omega)$ is symmetric with respect to the measure 
\[
\lambda_\alpha = \alpha \lambda_\Omega + (1-\alpha) \lambda_\partial
\] 
on $\overline\Omega$, where $\lambda_\Omega$ and $\lambda_\partial$ denote the normalized volume and Hausdorff-measures on $\Omega$ and $\partial \Omega$, and  the parameter $\alpha \in [0,1]$ is determined by
\[ \gamma = \frac \alpha {1-\alpha} \frac {|\partial \Omega|}{|\Omega|} .
\] 

Hence, the first nonzero eigenvalue/spectral gap is characterized by the Rayleigh quotient 
\[ \sigma_{\alpha,\beta} = \inf_{\substack{f \in C^1 (\overline \Omega)\\ \Var_{\lambda_\alpha}(f)>0}} \frac{\mathcal E_{\alpha, \beta}(f)}{\Var_{\lambda_\alpha}f},\]
where 
\[ \Var_{\lambda_\alpha}f =\int_{ \overline\Omega} f^2d\lambda_\alpha-\left(\int_{ \overline\Omega} fd\lambda_\alpha\right)^2\]
and 
\[ \mathcal E_{\alpha, \beta}(f) = \alpha \int_{\Omega} \|\nabla f\|^2 d \lambda_\Omega + (1- \alpha)  \int_{\partial \Omega} \beta \|\bnabla f\|^2 d \lambda_\partial, \]
 and $\bnabla$ denotes the tangential derivative operator on $\partial \Omega$.  
 This representation of $\sigma_{\alpha,\beta}$  formally interpolates between the two marginal cases of the spectral gap for reflecting Brownian motion on $\Omega$ when $\alpha=1$ and tangential  Brownian motion on  $\partial \Omega$  (with diffusion rate $\beta$) when $\alpha=0$. In the sequel we will consider only the case of $\beta =1$ for simplicity. Using Rellich embedding and the classical Agmon-Douglas-Nirenberg theory \cite{Agmon:1959} we find a unique smooth minimizer $f$.

\subsection{Estimate from Reilly Formula}

The goal of this section is to proof the lower bound of $\sigma_\alpha$ by $\frac{1}{M_{2,\alpha}}$ claimed in Theorem~\ref{thm_main}.  
Our proof is based on Reilly's formula (see~\cite{reilly})
\begin{equation}\label{reilly}
\begin{aligned}
\int_\Omega \left((\Delta f)^2 - \|\nabla^2 f \|^2 \right) dx
&=\int_\Omega \Ric (\nabla f, \nabla f) dx \\
&\quad
+\int_{\partial\Omega} \left( H ( \frac{\partial f}{\partial \nu})^2   +\mathrm{II} (\bnabla f, \bnabla f) +2 \Bdelta  f \frac{\partial f}{\partial \nu}\right)  dS
\end{aligned}
    \end{equation}
where $dx$ and $dS$ denote the Riemannian volume resp. surface measure on $\Omega$ and $\partial \Omega$, $\nabla^2 f$ is the Hessian of $f$ and $H$ is the mean curvature of $\partial\Omega$ (\textit{i.e.} the trace of $\mathrm{II}$).

\begin{proposition}
Under the assumptions of Theorem~\ref{thm_main}, it holds that 
\begin{align}
    \label{def_M2}
 \sigma_\alpha \geq \min\left\{ \frac \alpha {1-\alpha}  \frac {d k_2 }{3 d -1 } \frac {|\partial \Omega|}{|\Omega|},\frac {d}{d-1}k_R\right\}. 
\end{align}
\end{proposition}

\begin{proof}  
Let $f$ be an eigenfunction of $-A$, i.e.  \[
\begin{cases}
\Delta f + \sigma f = 0 & \mbox{ in } \Omega\\
\Bdelta f - \gamma \frac{\partial f}{\partial \nu} + \sigma f =0 & \mbox{ on } \partial \Omega,
\end{cases}\]
where $\gamma = \frac \alpha {1-\alpha} \frac{|\partial \Omega|}{|\Omega|}$.  We apply  Reilly's formula~\eqref{reilly} to the corresponding eigenfunction $f$. In this case, for the l.h.s. we estimate 
\begin{align*}
\int_\Omega \left((\Delta f)^2 - \|\nabla^2 f \|^2 \right) dx
& \leq \frac{d-1}{d} \int_\Omega (\Delta f)^2 dx= - \frac {d-1}{d}\sigma \int_\Omega f \Delta f dx \\
& = \frac{d-1}{d}\sigma \int_\Omega \|\nabla f\|^2   dx - \frac {d-1}{d}\sigma \int_{\partial\Omega}  \frac{\partial f}{\partial \nu}f dS  \\
& = \frac{d-1}{d}\sigma \int_\Omega \|\nabla f\|^2   dx -
 \frac {d-1}{d}\frac{\sigma}{\gamma} \int_{\partial\Omega} (\Bdelta f +  {\sigma} f) f dS\\
& = \frac{d-1}{d}\sigma \int_\Omega \|\nabla f\|^2   dx + 
 \frac {d-1}{d}\frac{\sigma}{\gamma} \int_{\partial\Omega} \|\bnabla f\|^2  dS -  \frac {d-1}{d}\frac{\sigma^2 }{\gamma} \int_{\partial\Omega}  f^2   dS\\
 & \leq   \frac{d-1}{d}\sigma \int_\Omega \|\nabla f\|^2   dx + 
 \frac {d-1}{d}\frac{\sigma}{\gamma} \int_{\partial\Omega} \|\bnabla f\|^2  dS. \end{align*}
Since  
\begin{align*}
\int_{\partial\Omega} \frac{\partial f}{\partial \nu} \Bdelta f dS & = \frac  1 \gamma \int_{\partial\Omega} (\Bdelta f + \sigma f) \Bdelta f dS\\
& = \frac 1  \gamma \int_{\partial\Omega}(\Bdelta f)^2 dS - \frac \sigma \gamma \int_{\partial\Omega} \|\bnabla f \|^2 dS
\geq - \frac \sigma \gamma \int_{\partial\Omega} \|\bnabla f \|^2 dS,
\end{align*}
the r.h.s. of~\eqref{reilly} is bounded from below by 
\begin{align*} 
 k_R \int_\Omega  \|\nabla f\|^2 dx    &- \frac {2 \sigma}  \gamma \int_{\partial\Omega} \|\bnabla f\|^2 dS + \int_{\partial\Omega} h \left|\frac{\partial f}{\partial \nu}\right|^2 dS + k_2  \int_{\partial\Omega} \|\bnabla f\|^2 dS
\\
& \geq k_R \int_\Omega \|\nabla f\|^2 dx   - \frac {2 \sigma}  \gamma \int_{\partial\Omega} \|\bnabla f\|^2 dS + k_2 \int_{\partial\Omega}  \|\bnabla f\|^2 dS. 
 \end{align*}
Combining the two bounds  for \eqref{reilly} yields 
\[ \left(\frac {d-1}{d } \sigma - k_R \right) \int_\Omega \|\nabla f\|^2 dx \geq \left(k_2-\frac {3d-1}{d} \frac \sigma \gamma \right) \int_{\partial\Omega} \|\bnabla f\| ^2 dS,\]
which implies that either 
\[ k_2 - \frac {3d-1}d \frac {\sigma}\gamma \leq 0,\quad \mbox{ i.e. }\quad  \sigma \geq \frac {d k_2 \gamma }{3 d -1 }\]
or 
\[ \frac{d-1}{d}\sigma - k_R \geq 0,\quad \mbox{ i.e. }\quad\sigma  \geq  \frac {d}{d-1}k_R,\]
which yields the claim by inserting the definition of $\gamma$. 
\end{proof}

\section{Approach via Interpolation}
\label{sec:interp_method}

\subsection{Energy Decomposition}\label{subsec:energy_decomposition}
The previous estimate via direct application of Reilly formula gives weak lower bounds for small $\alpha$, in case of small boundary, or $\Omega$ having almost flat parts.  One might ask for an alternative approach, which uses also (Neumann) spectral information on $\Omega$ and $\partial \Omega$. For later use it will be convenient to work with a slight generalization of the setup above. To this aim let $\Omega$ be an  open domain   in $\R^d$ or a Riemannian manifold with a piecewise smooth boundary  $\partial \Omega$.
Let $\Sigma$ be a smooth compact and connected subset of $\partial \Omega$. In most of the results below we will take $\Sigma=\partial\Omega$; the case $\Sigma\subsetneq\partial\Omega$ is mainly relevant when discussing the continuity of $C_{\alpha}$ and the case of a ball with needle.
We denote by $\partial \Sigma$ the boundary of $\Sigma$ in the space $\partial \Omega$, \textit{i.e.} $\partial \Sigma=\Sigma \cap \overline{\partial \Omega \backslash \Sigma}$. 
We consider  two probability measures $\lambda_\Omega$ and $\lambda_\Sigma$ with support $\Omega$ and $\Sigma$, which are absolutely continuous with respect to the Lebesgue and  the Hausdorff measures on $\Omega$ and  $\Sigma$, respectively.

Let $\Deriv: C^1(\Omega) \mapsto \Gamma^0(\Omega)$ and $\bDeriv: C^1(\partial \Omega) \mapsto \Gamma^0(\partial \Omega)$ denote given first order gradient operators mapping differentiable functions into the set $\Gamma^0(\Omega)$ resp.\ $\Gamma^0(\partial \Omega)$ of continuous  (tangential) vector fields  on $\Omega$ and  on $\partial \Omega$, and for $\alpha \in [0,1]$ let \begin{align*}
\lambda_\alpha&:= \alpha \lambda_\Omega + (1-\alpha) \lambda_\Sigma, \\
\mathcal E_\alpha (f) &:= \alpha \int_\Omega \| \Deriv f \|^2 d\lambda_\Omega+ (1-\alpha)  \int_{\Sigma} \| \bDeriv f \|^2 d\lambda_\Sigma, \quad  f \in \DomE_0,
\end{align*}
where $\DomE_0 \subset   \mC^1(\overline \Omega)$ is  dense in $C(\overline\Omega)$. We assume that  for $\alpha \in [0,1]$ the quadratic form $(\mathcal E_\alpha,\DomE_0)$ is a pre-Dirichlet form on $L^2(\overline \Omega,\lambda_\alpha)$ whose closure we shall denote by $(\mathcal E_\alpha,\DomE)$, c.f. \cite{grothaus} for details. We wish to estimate from above $\sigma_\alpha^{-1}=C_\alpha$, where $C_\alpha$ is the optimal Poincaré constant given by
\begin{equation}
  \label{equ_constant_c_alpha}
  C_\alpha:= \sup_{\substack {f \in \DomE_0 \\ \mathcal E_\alpha(f)>0}}  \frac{\Var_{\lambda_\alpha} f}{\mathcal E_\alpha (f)} .
\end{equation}
In the interpolation method presented below  it is assumed that $C_\alpha$ are known or can be estimated at the two extremals  $\alpha\in \{0,1\}$. For instance, when $\Deriv= \nabla$, $\bDeriv = \bnabla$ are the standard gradient resp.\ tangential gradient operators and $\lambda_\Omega$ and $\lambda_\Sigma$ are normalized Lebesgue resp. Hausdorff measures on $\Omega$ and $\Sigma \subset \partial \Omega$,   $C_\Omega:=C_1$ is the optimal Poincaré constant associated to the Laplace operator on $\Omega$ with Neumann boundary conditions, 
whereas  $C_\Sigma:=C_0$ is  the optimal Poincaré constant associated to the Laplace-Beltrami operator on $\Sigma$ with Neumann boundary conditions on $\partial\Sigma$. 

\smallskip 

The following proposition establishes an estimate of $C_\alpha$ in terms of $C_\Omega$ and $C_\Sigma$, where $C_\Omega$ and $C_\Sigma$ are upper bounds for the 
Poincaré constants $C_1$ and $C_0$, respectively. 
\begin{proposition}
\label{prop:interp_method}
Assume  there exists constants $\Cint$, $K_1, K_2$ such that for any $f \in \DomE_0$
\begin{align}
\Var_{\lambda_\Sigma} f
\leq  \Cint \int_\Omega \| \Deriv f \|^2 d\lambda_\Omega ,
\label{ineq Cint}
\end{align}
and 
\begin{align}
\left( \int_\Omega f d \lambda_\Omega - \int_\Sigma f d \lambda_\Sigma \right)^2
\leq K_1 \int_\Omega \| \Deriv f \|^2 d\lambda_\Omega + K_2 \int_\Sigma \| \bDeriv f \|^2 d\lambda_\Sigma,
\label{ineq K}
\end{align}
then it holds for any  $\alpha \in (0,1)$, 
\begin{align}
C_\alpha \leq
\max \left\{
C_\Omega + (1-\alpha)K_1 ,
\alpha K_2 ,
\frac{(1-\alpha)\Cint C_\Sigma +\alpha C_\Omega C_\Sigma+\alpha(1-\alpha)(\Cint K_2+C_\Sigma K_1)}{(1-\alpha) \Cint + \alpha C_\Sigma} \right\}.
\label{inequality_interpolation}
\end{align}
\end{proposition}

\begin{proof}
By definition of $C_\Sigma$ and by~\eqref{ineq Cint}, for any $f \in \DomE_0$
\begin{align*}
\Var_{\lambda_\Sigma} f \leq t \Cint \int_\Omega \| \Deriv f \|^2 d\lambda_\Omega + (1-t) C_\Sigma \int_\Sigma \| \bDeriv f \|^2 d\lambda_\Sigma,
\end{align*}
for any $t \in [0,1]$. 
Let $\alpha \in (0,1)$. For any $f  \in \DomE_0$ and any $t \in [0,1]$
\begin{align*}
\Var_{\lambda_\alpha} f 
&= \alpha \Var_{\lambda_\Omega} f + (1-\alpha) \Var_{\lambda_\Sigma} f
+\alpha(1-\alpha) \left( \int_\Omega f d \lambda_\Omega - \int_\Sigma f d \lambda_\Sigma \right)^2 \\
&\leq \left( C_\Omega +\frac{(1-\alpha)t}{\alpha} \Cint + (1-\alpha) K_1 \right) \alpha \int_\Omega \| \Deriv f \|^2 d\lambda_\Omega \\
&\quad+ \left( (1-t) C_\Sigma +  \alpha K_2 \right) (1-\alpha)\int_\Sigma \| \bDeriv f \|^2 d\lambda_\Sigma.
\end{align*}
Therefore,
\begin{align*}
C_\alpha \leq \inf_{t \in [0,1]} \max\left\{C_\Omega +\frac{(1-\alpha)t}{\alpha} \Cint + (1-\alpha) K_1 ,(1-t) C_\Sigma +  \alpha K_2 \right\}.
\end{align*}
For any positive constants $a,b,c,d$, we have
\begin{align*}
\inf_{t \in [0,1]} \max\left\{a+bt,c-dt\right\}
= 
\begin{cases}
a &\text{if } c-a<0,\\
c-d &\text{if } c-a>b+d,\\
\frac{bc+ad}{b+d} &\text{if } 0\leq c-a \leq b+d.
\end{cases}
\end{align*}
Therefore
\begin{equation}\label{eq_estimate_of_C_alpha}
C_\alpha \leq
\begin{cases}
C_\Omega + (1-\alpha)K_1\\ \smallskip 
\phantom{C_\Omega (1-\alpha) } \text{\ \ if }  \alpha K_2-(1-\alpha)K_1 + C_\Sigma-C_\Omega <0,\\ \smallskip 
\alpha K_2 \\ \smallskip
\phantom{C_\Omega  (1-\alpha)} 
 \text{\ \ if }  \alpha K_2-(1-\alpha)K_1 -C_\Omega >  \frac{1-\alpha}{\alpha}\Cint,\\ \smallskip
\frac{(1-\alpha)\Cint C_\Sigma +\alpha C_\Omega C_\Sigma+\alpha(1-\alpha)(\Cint K_2+C_\Sigma K_1)}{(1-\alpha) \Cint + \alpha C_\Sigma} \\ 
\phantom{C_\Omega (1-\alpha) }  
     \text{\ \ if } 0 \leq \alpha K_2-(1-\alpha)K_1 + C_\Sigma-C_\Omega 
     \leq C_\Sigma + \frac{1-\alpha}{\alpha}\Cint.
\end{cases}
\end{equation}
The last term is equivalent to the announced result. 
\end{proof}

\begin{remark}
In some examples considered later, we have $K_2 = 0$. 
In this case, the second alternative in the general estimate~\eqref{eq_estimate_of_C_alpha} is excluded. 
If moreover 
\begin{equation}\label{eq_third_case_in_estimate}
C_\Omega + (1-\alpha)K_1 \le C_\Sigma,
\end{equation}
then the third alternative applies and we further estimate
\begin{equation}\label{eq_estimate_of_third_case}
\begin{split}
&\frac{(1-\alpha)\Cint C_\Sigma + \alpha C_\Omega C_\Sigma 
      + \alpha(1-\alpha) C_\Sigma K_1}
     {(1-\alpha)\Cint + \alpha C_\Sigma}\\
&\qquad\qquad=
C_\Sigma 
\frac{(1-\alpha)\Cint + \alpha\bigl(C_\Omega + (1-\alpha)K_1\bigr)}
     {(1-\alpha)\Cint + \alpha C_\Sigma}
\le C_\Sigma,
\end{split}
\end{equation}
where the inequality follows from~\eqref{eq_third_case_in_estimate}. 
Consequently,
\[
C_\alpha \le \max\left\{ C_\Omega + (1-\alpha)K_1, C_\Sigma \right\}.
\]
This provides a simpler upper bound for $C_\alpha$. 
However, it is in general weaker than~\eqref{inequality_interpolation} for large $C_\Sigma$. 
For instance, when $\Omega$ is the unit ball 
(see Section~\ref{subsec:ball_full}), the original estimate yields a strictly better bound. Alternatively, we can estimate the right hand side of \eqref{eq_estimate_of_third_case} by $C_\Omega + (1-\alpha)K_1+\frac{1-\alpha}{\alpha}K_{\Sigma,\Omega}$. This leads to the inequality
\[
C_\alpha\le C_\Omega + (1-\alpha)K_1+\frac{1-\alpha}{\alpha}K_{\Sigma,\Omega}
\]
which also is weaker than~\eqref{inequality_interpolation} for $C_\Sigma$ close to $C_\Omega + (1-\alpha)K_1$.
\end{remark}

\subsection{Application to manifold case}

In this section, we complete the proof of Theorem~\ref{thm_main}. 
More precisely, we show that $C_\alpha \le M_{1,\alpha}$, where $C_\alpha = \sigma_\alpha^{-1}$.
This inequality  will be obtained via Proposition~\ref{prop:interp_method} 
and the subsequent two propositions. 
To this end, we specialize the general framework of 
Proposition~\ref{prop:interp_method} by taking 
$\Sigma = \partial \Omega$, $D = \nabla$, $D^\tau = \nabla^\tau$, 
$\lambda_\Sigma = \lambda_\partial$, and letting $\lambda_\Omega$ 
denote the normalized Lebesgue measure on $\Omega$.

\begin{proposition} \label{prop:K_one_manif}
Under the assumptions of Theorem~\ref{thm_main}, the inequality~\eqref{ineq K} holds with 
$K_2 = 0$ and
\begin{align*}
K_1 = \frac{d-1}{d k_R}.
\end{align*}
\end{proposition}

\begin{proof}
Our goal is to obtain an lower bound of
\begin{align*}
\inf_{\substack {f \in C^1(\overline \Omega) }}  \frac{\int_\Omega \| \nabla f \|^2 d\lambda_\Omega}{\left( \int_\Omega f d \lambda_\Omega - \int_{ \partial\Omega} f d \lambda_{ \partial} \right)^2},
\end{align*}
using Reilly's formula. We note  that
\begin{align*}
\inf_{\substack {f \in C^1(\overline \Omega) }}  \frac{\int_\Omega \| \nabla f \|^2 d\lambda_\Omega}{\left( \int_\Omega f d \lambda_\Omega - \int_{ \partial\Omega} f d \lambda_{\partial} \right)^2} 
=\inf_{\substack {f \in C^1(\overline \Omega)\\ \int_{ \partial\Omega} f d \lambda_{ \partial}=0  }}  \frac{\int_\Omega \| \nabla f \|^2 d\lambda_\Omega}{\left( \int_\Omega f d \lambda_\Omega  \right)^2} 
\geq
\inf_{\substack {f \in C^1(\overline \Omega)\\ \int_{ \partial\Omega} f d \lambda_{ \partial}=0  }}  \frac{\int_\Omega \| \nabla f \|^2 d\lambda_\Omega}{ \int_\Omega f^2 d \lambda_\Omega } =:\sigma. 
\end{align*}
Let $f \in C^1(\overline \Omega)$  be a minimizer for $\sigma$, which exists due to the standard argument based on Rellich-Kondrachov Theorem and the regularity property of solutions to elliptic equations (see e.g. \cite{Agmon:1959}). Then $\int_{ \partial\Omega} f d \lambda_{ \partial}=0$ and 
\begin{align*}
\int_\Omega \nabla f \cdot \nabla \xi d\lambda_\Omega
=\sigma \int_\Omega f \xi d\lambda_\Omega
\end{align*}
for each $\xi \in C^1(\overline \Omega)$ with $\int_{ \partial\Omega} \xi d \lambda_{ \partial}=0$. By integration by parts, the latter equality is equivalent to
\begin{align*}
-\int_\Omega \Delta f   \xi d\lambda_\Omega + \frac{|{ \partial\Omega}|}{|\Omega|} \int_{ \partial\Omega} \frac{\partial f}{\partial \nu} \xi d\lambda_{ \partial}
=\sigma \int_\Omega f \xi d\lambda_\Omega
\end{align*}
for each $\xi \in C^1(\overline \Omega)$ satisfying $\int_{ \partial\Omega} \xi d \lambda_{ \partial}=0$.
In particular, choosing $\xi \in C^\infty_0(\Omega)$ (which obviously satisfies $\int_{ \partial\Omega} \xi d \lambda_{ \partial}=0$), we get that $f$ should satisfy $-\Delta f=\sigma f$ in $\Omega$. Hence $ \int_{ \partial\Omega} \frac{\partial f}{\partial \nu} \xi d\lambda_{ \partial}=0$ for each $\xi$ with zero mean, so it follows that $\int_{ \partial\Omega} \frac{\partial f}{\partial \nu} \left(\xi-\int_{ \partial\Omega} \xi d \lambda_{ \partial}\right) d\lambda_{ \partial}=0$ for every $\xi \in C^1(\overline \Omega)$, which is equivalent to
\begin{align*}
\int_{ \partial\Omega} \left( \frac{\partial f}{\partial \nu}-\int_{ \partial\Omega} \frac{\partial f}{\partial \nu} d \lambda_{ \partial}\right)\xi d\lambda_{ \partial}=0
\end{align*}
for every $\xi \in C^1(\overline \Omega)$. It follows that $\frac{\partial f}{\partial \nu}$ is constant on ${ \partial\Omega}$. Therefore, $f$ satisfies
\begin{equation} 
\label{minimizer}
  \begin{cases}
    \Delta f= -\sigma f & \mbox{in}\ \ \Omega,\\
    \frac{\partial f}{\partial \nu} \equiv c & \mbox{on}\ \ \partial \Omega,\\
    \int_{ \partial\Omega} f d\lambda_{ \partial}=0,    
  \end{cases}
\end{equation}
for some constant $c$. 

Since $f$ satisfies~\eqref{minimizer}, 
\begin{align*}
\int_\Omega (\Delta f)^2 dx
= -\sigma \int_\Omega f\Delta f dx
&=\sigma \int_\Omega \|\nabla f\|^2 dx
-\sigma 
\int_{ \partial\Omega} \frac{\partial f}{\partial \nu} f dS \\
&=\sigma \int_\Omega \|\nabla f\|^2 dx
-\sigma 
c \int_{ \partial\Omega}  f dS =\sigma \int_\Omega \|\nabla f\|^2 dx,
\end{align*}
because $\int_{ \partial\Omega}  f dS=|{ \partial\Omega}| \int_{ \partial\Omega} f d\lambda_{ \partial}=0$.
Furthermore, note that $\|\nabla^2 f \|^2=\sum_{i,j} (\partial_{ij}^2 f)^2\geq \sum_{i=1}^d (\partial_{ii}^2 f)^2\geq \frac{1}{d}(\sum_{i=1}^d \partial_{ii}^2 f)^2=\frac{1}{d} (\Delta f)^2$. Therefore, the l.h.s. of Reilly's formula~\eqref{reilly} is bounded by
\begin{align*}
\int_\Omega \left((\Delta f)^2 - \|\nabla^2 f \|^2 \right) dx
\leq \frac{d-1}{d}\int_\Omega (\Delta f)^2  dx
\leq \frac{d-1}{d} \sigma \int_\Omega \|\nabla f\|^2 dx. 
\end{align*}
On the other hand, by the assumptions of Theorem~\ref{thm_main}, $H \geq 0$, $\mathrm{II} (\bnabla f, \bnabla f)\geq 0$ and 
\begin{align*}
\int_\Omega \Ric (\nabla f, \nabla f) dx 
\geq k_R \int_\Omega \|\nabla f\|^2 dx. 
\end{align*}
Since 
\begin{align*}
\int_{ \partial\Omega}  \Bdelta  f \frac{\partial f}{\partial \nu} dS=c\int_{ \partial\Omega}  \Bdelta  f  dS=0
\end{align*}
the r.h.s.\ of~\eqref{reilly} is bounded from below by $k_R \int_\Omega \|\nabla f\|^2 dx$. It turns out that
\begin{align*}
\frac{d-1}{d} \sigma \int_\Omega \|\nabla f\|^2 dx
\geq k_R \int_\Omega \|\nabla f\|^2 dx,
\end{align*}
which implies that $\sigma\geq \frac{d}{d-1} k_R$. It follows that inequality~\eqref{ineq K} holds with $K_1= \frac{d-1}{d k_R}$. 
\end{proof}

\begin{remark} Instead of using $K_1$ from Proposition~\ref{prop:K_one_manif} with $K_2=0$ another admissible choice  is \[K_1'  = \frac{|\Omega|}{|\partial \Omega|}B^2  (1+C_\Omega)<\infty,\] where $B$ is  the {optimal Sobolev trace constant} of $\Omega$, i.e.\ the norm of the embedding $H^{1,2}(\Omega)\hookrightarrow L^2 (\partial \Omega)$.
 $B^{-2}$ is the first nontrivial eigenvalue of a  Steklov-type eigenvalue problem 
 \begin{equation*}
\begin{cases}
  -\Delta f + f =0 &  \mbox{ in } \Omega \\
 \frac{\partial f}{\partial \nu}  = \sigma f  & \mbox{ on } \partial \Omega,
    \end{cases}\end{equation*}
for which however explicit lower bounds in terms of the geometry of $\Omega$ seem yet unknown \cite{MR1978428,MR1971310, MR3299025,MR1443055,MR2384747}. 
\end{remark}

\begin{proposition} \label{prop:Escobar_ineq}
Under the assumptions of Theorem~\ref{thm_main}, the inequality~\eqref{ineq Cint} holds with $\Cint=\frac{2}{k_2}$. 
\end{proposition}

\begin{proof} The optimal choice for $\Cint$ is $\sigma^{-1}$, where $\sigma$ is given by
\[
\sigma=\inf_{\substack {f \in C^1(\overline \Omega)\\ \int_{ \partial\Omega} f d \lambda_{ \partial}=0  }}  \frac{\int_\Omega \| \nabla f \|^2 d\lambda_\Omega}{ \int_{ \partial\Omega} f^2 d \lambda_{ \partial} }\]
is the first nontrivial eigenvalue of the Steklov-problem  c.f.\ \cite{MR3662010}
\begin{equation*}
\begin{cases}
 \Delta f = 0 &  \mbox{ in } \Omega, \\
 \frac{\partial f}{\partial \nu} ={ \frac{|\Omega|}{|\partial\Omega|}} \sigma f & \mbox{ on } \partial \Omega.
    \end{cases}\end{equation*}
Escobar \cite{escobar}  showed  $\sigma \geq \frac {k_2|\partial\Omega|} {2|\Omega|} $ in this case.
\end{proof}

Combining Proposition~\ref{prop:interp_method} with 
Propositions~\ref{prop:K_one_manif} and~\ref{prop:Escobar_ineq}, we get $\sigma_\alpha
\ge \frac{1}{M_{1,\alpha}}$, that completes the proof of Theorem~\ref{thm_main}.

\begin{remark}
When $\alpha$ goes to 0, $ M_{1,\alpha}$ tends to $\max\left\{C_\Omega+\frac{d-1}{dk_R},C_{\partial}\right\}$ and $M_{2,\alpha}$ tends to $+\infty$, so the estimation obtained via the interpolation method is always stronger. 
When $\alpha$ goes to $1$, $M_{1,\alpha}$ tends to $C_\Omega$ and $M_{2,\alpha}$ tends to $\frac{d-1}{dk_R}$, so the relative strength of each method depends on the values of $C_\Omega$, $d$ and $k_R$. 
\end{remark}

\section{Further Examples} \label{Sec:examples}
\subsection{Brownian motion on Euclidean balls with sticky  boundary diffusion}
\label{subsec:ball_full}

To show that the interpolation method can produce good results also in cases when $k_R=0$ we consider the example when   $\Omega:= B_1$ is the unit ball in $\R^d$, $\Sigma=\partial \Omega$ and  $\Deriv=\nabla$ and $\bDeriv=\sqrt \beta \, \bnabla$ with $\DomE_0=C^1(\overline \Omega)$.

\begin{proposition}
\label{prop:example_full_sphere}
In the case when $\Omega=B_1\subset \R^d$ the optimal  Poincaré constant of the generator~\eqref{eq:gendef} is bounded from above by 
\begin{align}
\label{example_full_sphere}
C_\alpha \leq
\max \left(
C_\Omega + (1-\alpha)\frac{d+1}{4d^2},
\frac{4(1-\alpha) d  +4\alpha d^2 C_\Omega +\alpha(1-\alpha)(d+1) }{4d(\alpha d + (1-\alpha) \beta (d-1)) }
\right),
\end{align}
where $\alpha=\frac{\gamma}{d+\gamma}$ and $C_\Omega$ is the optimal Poincaré constant for reflecting Brownian motion on $B_1\subset \R^d$.
\end{proposition}

\begin{proof}
In order to apply Proposition~\ref{prop:interp_method}, it is sufficient to compute the constants $C_\Sigma$,  $\Cint$, $K_1$ and $K_2$. We claim that inequalities~\eqref{ineq Cint} and~\eqref{ineq K} holds with
\[ C_\Sigma =\frac{1}{\beta(d-1)}, \quad \Cint=\frac{1}{d}, \quad K_1= \frac{d+1}{4d^2}, \quad K_2=0. \]

First, according to~\cite[Theorem 22.1]{Shubin}, the first eigenvalue of the Laplace-Beltrami operator on the unit sphere of dimension $d-1$ is equal to $d-1$, thus $C_\Sigma=\frac{1}{\beta (d-1)}$. 

Moreover, according to~\cite[Theorem~4]{Beckner:1993}, for every $f \in \mC^1(\partial \Omega)$ one has 
\begin{equation*}
  \left( \int_{ \partial \Omega  }   |f|^q d\lambda_{\Sigma }  \right)^{ \frac{2}{ q }}\leq \frac{ q-2 }{ d } \int_{ \Omega}   \|\nabla u\|^2 d\lambda_\Omega+\int_{ \partial \Omega  }   f^2 d\lambda_{\Sigma },  
\end{equation*}
for $2\leq q<\infty$ if $d=2$ and $2\leq q< \frac{ 2d-2 }{ d-2 }$ if $d\geq 3$, where $u$ is the harmonic extension of $f$ to  the unit ball $\Omega$. 
It implies the logarithmic Sobolev inequality $ \operatorname{Ent}_{\lambda_\Sigma} (f^2)\leq  \frac{2}{ d }\int_{ \Omega }   \|\nabla u\|^2 d\lambda_\Omega$.
Repeating the proof of Proposition~5.1.3 in~\cite{Bakry:2014}, we get $
\Var_{\lambda_{\Sigma }} f\leq  \frac{1}{ d }\int_{ \Omega }   \|\nabla u\|^2 d\lambda_\Omega$.
Moreover, since the harmonic extension of $f$ is minimizing the energy functional $\mathcal E_1$ under any function with boundary condition $f$, the last inequality implies for any $f \in \mC^1(\overline \Omega)$
\begin{align}
\label{equ_poincare_inequality_on_s}
\Var_{\lambda_{\Sigma }} f\leq  \frac{1}{ d }\int_{ \Omega }   \|\nabla f\|^2 d\lambda_\Omega,
\end{align}
which implies $\Cint=\frac{1}{d}$.

Furthermore, note that $\int_{\partial \Omega} f(y)  \lambda_\Sigma(dy)=\int_\Omega f(\pi_x)  \lambda_\Omega (dx)$, where $\pi_x= \frac{x}{ \|x\| }$, $x \not= 0$. Hence, using  Jensen's inequality and polar coordinates
\begin{align*}
 \left( \int_\Omega f d \lambda_\Omega - \int_{\partial \Omega} f d \lambda_\Sigma \right)^2&\leq  \int_{\Omega} (f(x)-f(\pi_x))^2  \lambda_\Omega (dx)\\
       &=  \frac{ 1 }{ |\Omega| }\int_{ \partial \Omega } \int_{ 0 }^{ 1 }  \left(f(ry)-f(y)\right)^2r^{d-1}drdy \\
    &=  \frac{ 1 }{|\Omega|} \int_{ \partial \Omega }   \int_{ 0 }^{ 1 }  \left( \int_{ r }^{ 1 } \frac{ d }{ ds }f(sy)ds\right)^2r^{d-1}drdy\\
  &\leq \frac{ 1 }{ |\Omega| } \int_{ \partial \Omega }    \int_{ 0 }^{ 1 } (1-r)\left( \int_{ r }^{ 1 } \left(\frac{d}{ds}f(sy)\right)^2ds  \right)r^{d-1} drdy\\
    &= \frac{1}{ |\Omega| }\int_{ \partial \Omega }   \int_{ 0 }^{ 1 } \left[ \int_{ 0 }^{ s } (1-r) r^{d-1}dr  \right] \left( \frac{d}{ ds }f(sy) \right)^2 dsdy.
  \end{align*}
  We separately estimate
  \begin{align*}
  \int_{ 0 }^{ s } (1-r) r^{d-1}dr =  \left( \frac{s}{d}-\frac{s^2}{d+1}  \right)s^{d-1}
  \leq \frac{d+1}{4d^2}s^{d-1}.
  \end{align*}
  for any $s \in [0,1]$.
 Hence,
  \begin{align}
  \left( \int_\Omega f d \lambda_\Omega - \int_{\partial \Omega} f d \lambda_\Sigma \right)^2&\leq \frac{ d+1 }{ 4d^2|\Omega| }\int_{ \partial \Omega }    \int_{ 0 }^{ 1 }\left( \nabla f(sy) \cdot y\right)^2s^{d-1} ds dy\notag\\
  &= \frac{ d+1 }{ 4d^2|\Omega| } \int_{ \partial \Omega }    \int_{ 0 }^{ 1 } \left\| \nabla f(sy)\right\|^2 s^{d-1} dsdy \notag\\
  &= \frac{ d+1 }{ 4d^2 }\int_{ \Omega }   \|\nabla f(x)\|^2 \lambda_\Omega(dx). 
  \label{ineq_interpolation_ball}
  \end{align}
  which implies  $K_1= \frac{d+1}{4d^2}$ and $K_2=0$. 
\end{proof}

For illustration, in $d=2, $ we compare the bound from Proposition~\ref{prop:example_full_sphere} for $\beta=1,\gamma>0$ to the  optimal constant $C_\alpha$  which will be computed numerically. To evaluate the bound \eqref{example_full_sphere}, note that  in this case   \begin{equation}
\label{def:gamma_star}
C_\Omega=\frac{1}{\sigma_{\Omega}}\approx\frac{1}{3.39}, 
\end{equation}
where $\sigma_{\Omega}$ is the smallest positive eigenvalue of the Laplace operator with Neumann boundary condition on the circle. It is given as the minimal positive solution to the equation $J_m'(\sqrt{\gamma})=0$, $m\in\N_0$, where $J_m$ is the Bessel function of the first kind of parameter $m$,  defined  by $J_m(x)=\frac{1}{\pi} \int_0^\pi\cos (mt-x \sin t) dt$, $x \geq 0$.  As a consequence, 
 inequality~\eqref{example_full_sphere} becomes  
\begin{align}
\label{example_full_sphere_dim2}
C_\alpha \leq
\frac{8 (1-\alpha)\sigma_{\Omega}+16\alpha+3\alpha(1-\alpha)\sigma_{\Omega}}
{8(1+\alpha)\sigma_{\Omega} }.
\end{align}

For the numerical computation of  $C_\alpha$ one notes that  the generator $A_\alpha$ associated with $\e_\alpha$ is defined on $ D(A_\alpha) \subset \mC^2(\overline{\Omega})$ as  
\[
  A_\alpha f =  \I_{\Omega}\Delta f+\I_{\partial \Omega}\left( \Bdelta f- \frac{ 2 \alpha }{ 1-\alpha }\frac{\partial f}{\partial \nu} \right),
\]
where $\Bdelta$  and $\frac{\partial}{\partial \nu}$ denote the Laplace-Beltrami operator  and the outer normal derivative on the circle $\partial \Omega$.
Hence, an eigenvector of $-A_\alpha$ for eigenvalue $\lambda \geq 0$ is a function $f\in D(A_\alpha)$ such that
\[
  A_\alpha f=-\lambda f \quad \mbox{in}\ \ \Omega.
\]
This equation is equivalent to the system of partial differential equations
\begin{equation*}
  \label{equ_equation_for_eigenvalues}
  \begin{cases}
    \Delta f= -\lambda f & \mbox{in}\ \ \Omega,\\
    \Bdelta f-\frac{2\alpha}{1-\alpha} \frac{\partial f}{\partial \nu}= -\lambda f & \mbox{on}\ \ \partial \Omega,
  \end{cases}
\end{equation*}
which by the continuity of $f$ can be rewritten  as
\begin{equation*}
  \label{equ_equation_for_eigenvalues_with_boundary_conditions}
  \begin{cases}
    \Delta f= -\lambda f & \mbox{in}\ \ \Omega,\\
    \Delta f=\Bdelta f-\frac{2\alpha}{1-\alpha} \frac{\partial f}{\partial \nu} & \mbox{on}\ \ \partial \Omega.
  \end{cases}
\end{equation*}
Passing to polar coordinates  $(x_1,x_2)=(r\cos\theta,r\sin\theta)\in\Omega$ in $d=2$ and separating variables, we obtain the set of eigenfunctions $\{f_{m,l}^c,f_{m,l}^s\}_{m,l \in \N_0}$,
\[
f_{m,l}^c(x_1,x_2)=J_m(\sqrt{\lambda_{m,l}}r)\cos(m\theta), \quad m,l \in\N_0,
\]
\[
f_{m,l}^s(x_1,x_2)=J_m(\sqrt{\lambda_{m,l}}r)\sin(m\theta), \quad m \in\N,\ \ l \in\N_0,
\]
where  $\lambda_{m,l}$, $l \in\N_0$, are countable family of positive solutions to the  equation
\begin{equation}
  \label{equ_equation_for_gamma}
  \sqrt{\lambda} J_m''(\sqrt{ \lambda })+\frac{ 1+ \alpha }{ 1-\alpha }J_m'(\sqrt{ \lambda })=0
\end{equation}
for every $m \in \N_0$. Since the family $\{f_{m,l}^c,\ m,l \in\N_0\}\cup\{f_{m,l}^s,\ m\in\N_0,\ l\in\N_0\}$ is dense in $L_2(\Omega,\lambda_\alpha)$  and the operator $A_\alpha$ is symmetric, the standard argument implies 
\begin{equation}
\label{optimal_constant}
C_\alpha=\frac{1}{\lambda_{\alpha,\star}},
\end{equation}
where $\lambda_{\alpha,\star}=\min\limits_{m,l\in\N_0}\lambda_{m,l}$. The resulting curves  are plotted in Figure~\ref{fig:full_sphere}.

\begin{figure}[h]\label{fig:graphs}
\centering
  \includegraphics[width=83mm]{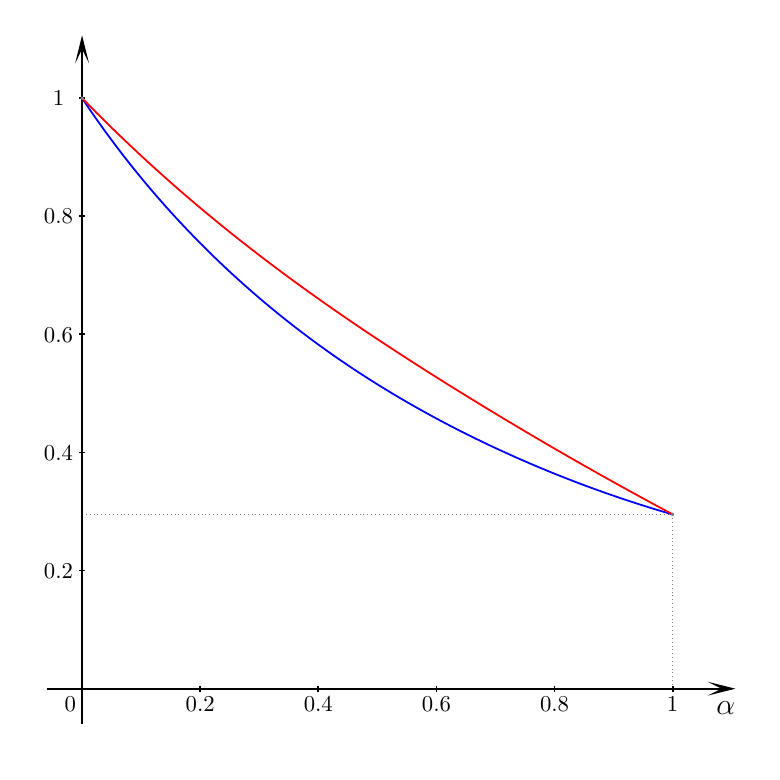}
  \caption{The blue curve represents $\alpha \mapsto C_\alpha$ the optimal Poincaré constant when $\Omega$ is the unit ball of $\R^2$ with full boundary diffusion. The red curve is the upper estimate given by~\eqref{example_full_sphere_dim2}.}
  \label{fig:full_sphere}
\end{figure}

\subsection{Brownian motion on balls with partial sticky reflecting boundary diffusion}
\label{subsec:partial_ball}

The case of sticky diffusion along only parts of the boundary and standard reflection elsewhere at the boundary can be treated in a similar way. As in Section~\ref{subsec:ball_full}, let  $\Omega:= B_1$ be the unit ball of $\R^2$  and $\beta=1$.  Now, define for a fixed 
 $\delta \in (0,1)$  
\[
\Sigma=\{ (\cos \theta, \sin \theta) \in \partial\Omega: -\delta \pi \leq \theta \leq \delta \pi\}, \quad \quad \neumann:=\partial \Omega \backslash \Sigma.
\]

\begin{proposition}
\label{prop_partial_sphere} It holds that 
\begin{align}
\label{estim_partial}
C_\alpha \leq
\max \left( 
C_\Omega + (1-\alpha)K_1(\delta), 
\frac{4(1-\alpha) \delta^2+8 \alpha \delta^3 C_\Omega +8 \alpha (1-\alpha) \delta^3 K_1(\delta)  }{(1-\alpha) + 8 \alpha \delta^3 }
\right),
\end{align}
where $C_\Omega = \frac{1}{\sigma_{\Omega}}\approx\frac{1}{3.39}$  and $K_1(\delta)= \left( \sqrt{1-\delta}\pi + \frac{1}{4}\sqrt{\frac{3}{\delta}} \right)^2$.
\end{proposition}

As previously, we will start by computing the needed constants $C_\Omega$, $C_\Sigma$, $\Cint$, $K_1$ and $K_2$. 
The first constant, $C_\Omega=\frac{1}{\sigma_{\Omega}}\approx \frac{1}{3.39}$, remains unchanged.
\begin{lemma} The following inequalities hold true
\begin{align}
\label{lem_ine_1}
\Var_{\mus} f &\leq C_\Sigma \int_{\Sigma} \|\bnabla f\|^2 d \mus, \\
\Var_{\mus} f &\leq \Cint \int_{\Omega} \|\nabla f\|^2 d\lambda_\Omega,
\label{lem_ine_2}
\end{align}
where $C_\Sigma=4\delta^2$ and $\Cint=\frac{1}{2\delta}$. 
\end{lemma}

\begin{proof}
Inequality~\eqref{lem_ine_1} corresponds to the Poincaré inequality of the Laplacian on the one-dimensional interval $[-\delta \pi, \delta \pi]$ with Neumann boundary conditions. It is well known (see~\cite[Proposition~4.5.5]{Bakry:2014}) that the optimal Poincaré constant is given by $C_\Sigma=4\delta^2$.

Moreover, let us decompose the normalized Hausdorff measure $\lambda_\partial$ on the sphere $\partial \Omega$ into the normalized Hausdorff measure $\mus$ on $\Sigma$  and the normalized Hausdorff measure $\mun$ on $\neumann$: $\lambda_\partial = \delta \mus + (1-\delta) \mun$. 
Therefore
\begin{align*}
\Var_{\lambda_{\partial }} f = \delta\Var_{\mus} f + (1-\delta)\Var_{\mun} f
  +  \delta(1-\delta) \left( \int_{\Sigma} f d\mus- \int_{\neumann} f d\mun \right)^2
  \geq \delta\Var_{\mus} f,
\end{align*}
Furthermore, recall that by inequality~\eqref{equ_poincare_inequality_on_s}, for any $f \in \mathcal C^1(\overline \Omega)$, $
\Var_{\lambda_{\partial }} f\leq  \frac{1}{ 2 }\int_{ \Omega }   \|\nabla f\|^2 d\lambda_\Omega$. 
It implies~\eqref{lem_ine_2}.
\end{proof}

\begin{lemma}
It holds that 
\begin{align*}
\left( \int_\Omega f d \lambda_\Omega - \int_{\Sigma} f d \mus \right)^2
\leq K_1(\delta) \int_\Omega \| \nabla f \|^2 d\lambda_\Omega
\end{align*}
with  $K_1(\delta)= \left( \sqrt{1-\delta}\pi + \frac{1}{4}\sqrt{\frac{3}{\delta}} \right)^2$.
\end{lemma}

\begin{proof}
For every $x \in \Omega\backslash \{0\}$ with polar coordinates $(r, \theta)$, $r \in (0,1)$, $\theta \in (-\pi,\pi]$, denote by $p_x$ the point of coordinates $(1, \delta \theta)$ on $\Sigma$. Obviously, $\int_{\Sigma} f(y)  \mus (dy)= \int_\Omega f(p_x)  \lambda_\Omega (dx)$ and by Jensen's inequality
\begin{align*}
I:=\left( \int_\Omega f d \lambda_\Omega - \int_{\Sigma} f d \mus \right)^2
\leq  \int_\Omega \left(f(x) -f(p_x)\right)^2  \lambda_\Omega (dx).
\end{align*}
Define $g(r, \theta):=f(r \cos(\theta), r \sin (\theta))$. Then
\begin{align}
\label{triangle}
I &\leq \frac{1}{\pi} \int_0^1 \int_{-\pi}^\pi (g(r, \theta)-g(1,\delta \theta))^2 r dr d\theta \leq  ( \sqrt{J_1} +  \sqrt{J_2})^2 ,
\end{align}
where $J_1=\frac{1}{\pi} \int_0^1 \int_{-\pi}^\pi (g(r, \theta)-g(r,\delta \theta))^2 r dr d\theta$ and $J_2=\frac{1}{\pi} \int_0^1 \int_{-\pi}^\pi (g(r, \delta\theta)-g(1,\delta \theta))^2 r dr d\theta$. 
On the one hand
\begin{align}
J_1 &= \frac{1}{\pi} \int_0^1 \int_{-\pi}^\pi \left(\int_{\delta \theta} ^\theta \frac{\partial g}{\partial \theta} (r,u) du\right)^2 r dr d\theta \leq \frac{1-\delta}{\pi} \int_0^1 \int_{-\pi}^\pi |\theta| \int_{-\pi}^\pi  \left(\frac{\partial g}{\partial \theta} \right)^2(r,u) du \; r dr d\theta \notag\\
&\leq (1-\delta)\pi^2  \frac{1}{\pi}\int_0^1  \int_{-\pi}^\pi  \left(\frac{1}{r}\frac{\partial g}{\partial \theta} \right)^2(r,u) du \; r dr 
\leq (1-\delta)\pi^2 \int_\Omega \| \nabla f \|^2 d\lambda_\Omega.
\label{ineq_J1}
\end{align}
On the other hand
\begin{align*}
J_2 
&\leq \frac{1}{\pi} \int_0^1 \int_{-\pi}^\pi (1-r)  \int_r ^1 \left(\frac{\partial g}{\partial r}\right)^2 (s,\delta \theta) ds \; r dr d\theta \leq \frac{1}{\pi} \int_0^1 \int_{-\pi}^\pi  \left(\frac{\partial g}{\partial r}\right)^2 (s,\delta \theta) \int_0^s (1-r) r dr ds d\theta.
\end{align*}
For every $s \in [0,1]$, 
$\int_0^s (1-r) r dr = \frac{s^2}{2}-\frac{s^3}{3}\leq \frac{3s}{16}$, thus
\begin{align}
J_2
&\leq \frac{3}{16\delta \pi} \int_0^1 \int_{-\delta \pi}^{\delta\pi}  \left(\frac{\partial g}{\partial r}\right)^2 (s,u) s ds du \leq \frac{3}{16\delta} \int_\Omega \| \nabla f \|^2 d\lambda_\Omega.
\label{ineq_J2}
\end{align}
The proof of the lemma is completed by putting together~\eqref{triangle}, \eqref{ineq_J1} and~\eqref{ineq_J2}.
\end{proof}

\begin{proof}[Proof of Proposition~\ref{prop_partial_sphere}]
We apply Proposition~\ref{prop:interp_method} with $C_\Omega = \frac{1}{\sigma_{\Omega}}$, $C_\Sigma=4\delta^2$, $\Cint=\frac{1}{2\delta}$, $K_1(\delta)= \left( \sqrt{1-\delta}\pi + \frac{1}{4}\sqrt{\frac{3}{\delta}} \right)^2$ and $K_2=0$. 
\end{proof}

\begin{figure}[ht]
\centering
\subfloat[$\delta=0.5$]{\includegraphics[width=70mm]{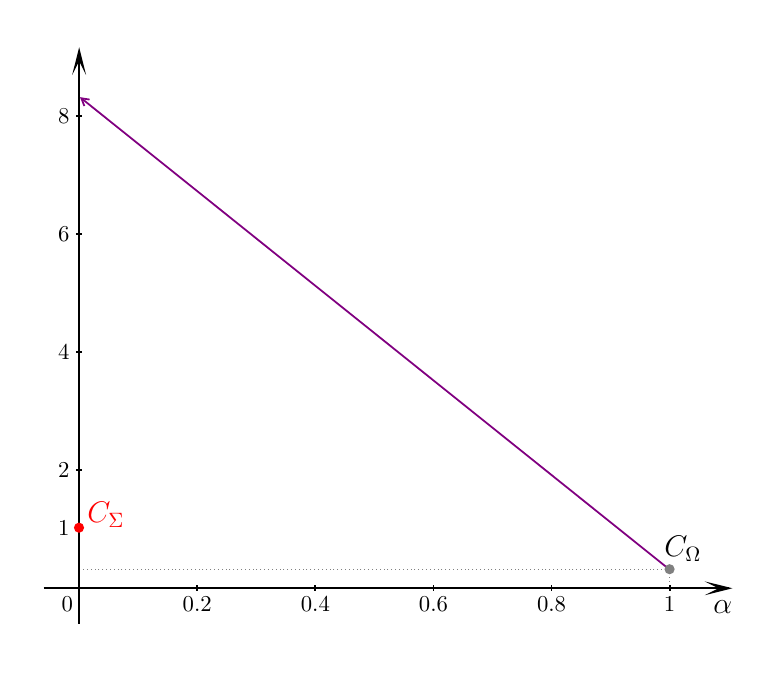}\label{delta1}}
\subfloat[$\delta=0.9$]{\includegraphics[width=70mm]{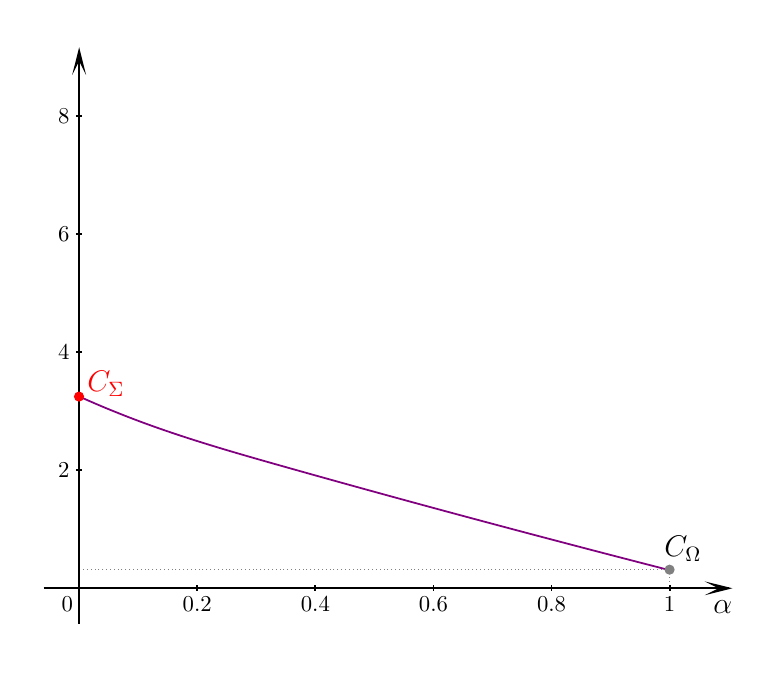}\label{delta2}}
\caption{The above two figures show the upper estimate given by the r.h.s of~\eqref{estim_partial}. 
In the case $\delta=0.9$ (Figure~\ref{delta2}), the curve interpolates between the extremal constants $C_\Sigma$ and $C_\Omega$, as opposed to the half-sphere case (Figure~\ref{delta1}).}
\label{fig:partial ball}
\end{figure}

For $\delta$ sufficiently large, the map $\alpha \mapsto C_\alpha$ is continuous at $\alpha=0$. Indeed, by Proposition~\ref{prop:continuity_condition}, a sufficient condition is $C_\Sigma(\delta) > C_\Omega + K_1(\delta)$, that is
\[4 \delta^2 > \frac{1}{\sigma_{\Omega}}+ \left( \sqrt{1-\delta}\pi + \frac{1}{4}\sqrt{\frac{3}{\delta}} \right)^2, \]
which is satisfied for any $\delta \geq 0.862$.

\subsection{Stratified spaces -- Ball with a needle}
\label{subsec:needle}

 Our final example presents an extension to stratified spaces, here in case of a  unit ball $\Omega=B_1$  of $\R^2$ with a needle $\mathcal L$   of length $L$  attached to one point of the boundary, i.e.\  $\mathcal L:=\{ (x,0): 1\leq x \leq L+1\}$, see Figure~\ref{fig:needle}.
The attachment point and the endpoint of the needle are denoted by $x_0:=(1,0)$ and $x_L=(L+1,0)$, respectively.

\begin{figure}[ht]
\centering
  \includegraphics[width=70mm]{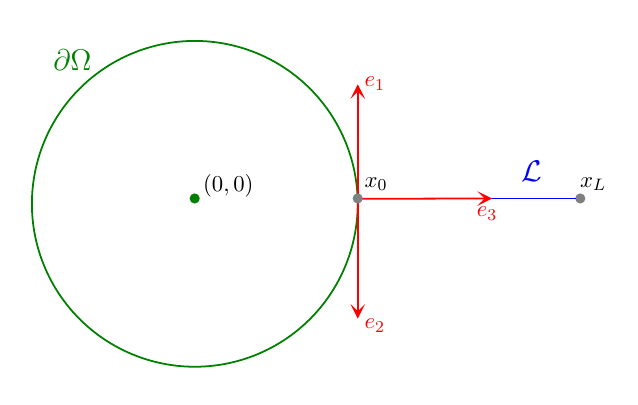}
  \caption{The ball (in green) is denoted by $\Omega$, the boundary of the ball is denoted by $\partial \Omega$ and the needle (in blue) is denoted by $\mathcal L$.}
  \label{fig:needle}
\end{figure}

In that setting, we define $\overline \Omega = \overline{ B_1 }\cup \mathcal L $, $\Sigma = \partial B_1 \cup \mathcal L$ and 
\[
\lambda_\alpha = \alpha \lambda_\Omega + (1-\alpha) \lambda_\Sigma,\]
where $\lambda_\Omega$ is as previously the normalized Lebesgue measure on $\Omega$ and $\lambda_\Sigma=\frac{2 \pi}{2 \pi+L} \lambda_\partial + \frac{L}{2\pi +L} \lambda_\mathcal L$, with $\lambda_\partial$ and $\lambda_\mathcal L$ being the normalized Hausdorff measures on $\partial \Omega$ and $\mathcal L$, respectively. We choose \[
 \DomE_0 =\left\{ f \in C_0(\overline{\Omega})) \cap C^1(\overline \Omega \setminus \{x_0\} ): \,   \frac{\partial f}{\partial e_1}+\frac{\partial f}{\partial e_2}+\frac{\partial f}{\partial e_3}= 0 \mbox{ at } x_0\right\},\]
 where $e_1=(0,1)$, $e_2=(0,-1)$ and $e_3=(1,0)$ are the three "tangent" vectors to $\Sigma$ at point $x_0$,
 and  $\Deriv := \nabla$, $\bDeriv:=\sqrt{\beta}\nabla^\tau$, which is well defined  in $\Sigma \setminus \{x_0\}$. With this choice, for $\alpha \in [0,1]$ $(\mathcal E_\alpha, \DomE_0)$ is a pre-Dirichlet form on $L^2(\overline \Omega, \lambda_\alpha) $, whose closure generates Brownian motion on $\Omega$ with sticky boundary diffusion on $\Sigma$, i.e.\ whose generator is given by  \[A_\alpha(f) =  \Delta f \one_{\Omega} + \beta \Delta_\Sigma f\one_{\Sigma} -\frac{\alpha}{1-\alpha}\frac{2\pi+L}{\pi} \frac {\partial f}{\partial \nu }\one_{\partial \Omega },\]
with $\Delta_\Sigma$ being the generator of the canonical diffusion on $\Sigma$ with reflecting boundary condition at $x_L$. As before, the optimal Poincaré constant $C_\alpha$  for $A_\alpha$ is given by
\begin{align*}
C_\alpha:=\sup_{\substack {f \in \DomE_0 \\ \mathcal E_\alpha(f)>0}}  \frac{\Var_{\lambda_\alpha} f}{\mathcal E_\alpha (f)},
\end{align*}
and  let $C_\Omega:=C_1$ and $C_\Sigma:=C_0$.  In this case the following estimate is obtained.
\begin{proposition}
\label{prop:ball_needle} Under the above assumptions,
\begin{align*}
    C_\alpha \leq \max \left( \frac{1}{\sigma_{\Omega}} + \frac{3}{8}(1-\alpha), \frac{1}{\beta\gamma_L} + \alpha \frac{L^2(\pi+L)}{\beta(2\pi + L)} \right),
\end{align*}
where $\gamma_L >0$ is the smallest positive solution to 
  \begin{align}
  \label{first eigenvalue needle}
 2 \cos (\sqrt{\gamma}L) (1- \cos (\sqrt{\gamma}2 \pi)) +\sin (\sqrt{\gamma}L)\sin (\sqrt{\gamma}2 \pi)=0.
 \end{align}
Note that $\gamma_L \leq 1$ for any $L>0$ and if  $L=2\pi$,  $\gamma_{2\pi}= \left(\frac{ \arccos(-1/3)}{2\pi}\right)^2 \approx 0.0925$. 
\end{proposition}

Let us compute the constants needed to apply Proposition~\ref{prop:interp_method}. 
As we do not expect an inequality of type~\eqref{ineq Cint} to hold in that case, we set $\Cint:=+\infty$. Moreover, $C_\Sigma$ can be computed exactly as follows.

\begin{lemma}
Under the assumptions above, the equality $C_\Sigma=\frac{1}{\beta \gamma_L}$ holds.
\end{lemma}

\begin{proof}
The constant $\frac{1}{C_\Sigma}$ is the smallest non-zero eigenvalue $\gamma$ of the  following problem:
\begin{align*}
  \begin{cases}
   \beta \Bdelta f = -\gamma f & \mbox{on}\ \ \Sigma \backslash\{x_0\},\\
    \frac{\partial f}{\partial \nu} = 0 & \mbox{at point }x_L,\\
    \frac{\partial f}{\partial e_1}+\frac{\partial f}{\partial e_2}+\frac{\partial f}{\partial e_3}= 0 & \mbox{at point }x_0,
  \end{cases}
\end{align*}
where $\Bdelta$ is the Laplace-Beltrami operator on $\partial \Omega $ and $\mathcal L$.
A general solution to that boundary value problem is given by
\begin{align*}
f(x)=\begin{cases}
 A \cos (\sqrt{\frac{\gamma}{\beta}}y)+B \sin (\sqrt{\frac{\gamma}{\beta}}y)  & \mbox{if } x=(y,0) \in \mathcal L, \\
C \cos (\sqrt{\frac{\gamma}{\beta}}\theta) + D \sin(\sqrt{\frac{\gamma}{\beta}}\theta)  & \mbox{if } x=(\cos \theta, \sin \theta) \in \partial \Omega,
\end{cases}
\end{align*}
where $A$, $B$, $C$ and $D$ have to satisfy the continuity assumption of $f$ at point $x_0$ and  both boundary conditions, that is:
\begin{align*}
\left\{ \begin{aligned}
A &= C = C \cos (\sqrt{\frac{\gamma}{\beta}} 2 \pi) + D \sin(\sqrt{\frac{\gamma}{\beta}}2 \pi) ,\\
0 &= -A \sin  (\sqrt{\frac{\gamma}{\beta}}L) + B \cos (\sqrt{\frac{\gamma}{\beta}}L), \\
0 &= B + D + C \sin (\sqrt{\frac{\gamma}{\beta}}2 \pi) - D \cos (\sqrt{\frac{\gamma}{\beta}}2 \pi).
\end{aligned}
\right.
\end{align*}
A short computation shows that this system has a non-trivial solution if and only if $\frac{\gamma}{\beta}$ solves~\eqref{first eigenvalue needle}. Therefore, $\frac{1}{C_\Sigma}=\beta \gamma_L$. Obviously, $\gamma=1$ is a solution to~\eqref{first eigenvalue needle}, thus $\gamma_L \leq 1$. 
 \end{proof}

Next, we look for the constants $K_1$ and $K_2$.

\begin{lemma}
Inequality~\eqref{ineq K} holds with $K_1=\frac{3}{8}$ and $K_2=\frac{L^2(\pi+L)}{\beta(2\pi + L)}$. 
\end{lemma}

\begin{proof}
Recall that $\Sigma=\partial \Omega \cup \mathcal L$. 
Let us insert the average of $f$ over $\partial \Omega$ as follows:
\begin{align*}
\left( \int_\Omega f d \lambda_\Omega - \int_\Sigma f d \lambda_\Sigma \right)^2
&\leq
2 \left( \int_\Omega f d \lambda_\Omega - \int_{\partial \Omega} f d \lambda_\partial \right)^2
+ 2 \left( \int_{\partial \Omega} f d \lambda_\partial - \int_\Sigma f d \lambda_\Sigma \right)^2 \\
&\leq \frac{3}{8} \int_\Omega \| \nabla f \|^2 d\lambda_\Omega
+ 2 \left( \int_{\partial \Omega} f d \lambda_\partial - \int_\Sigma f d \lambda_\Sigma \right)^2,
\end{align*}
where the second inequality follows directly from~\eqref{ineq_interpolation_ball}. 
Moreover, recalling that $\lambda_\Sigma=\frac{2 \pi}{2 \pi+L} \lambda_\partial + \frac{L}{2\pi +L} \lambda_\mathcal L$
\begin{align*}
\left( \int_{\partial \Omega} f d \lambda_\partial - \int_\Sigma f d \lambda_\Sigma \right)^2
 &=  \frac{L^2}{(2\pi +L)^2} \left( \int_{\partial \Omega} f d \lambda_\partial - \int_{\mathcal L} f d \lambda_\mathcal L \right)^2 .
\end{align*}
For every $x=(\cos\theta, \sin \theta) \in \partial \Omega$, with $\theta \in (-\pi,\pi]$, we denote by $p_x$ the point of $\mathcal L$ with coordinates $(1+L -\frac{|\theta| L}{\pi},0)$. It follows that
\begin{align*}
\left( \int_{\partial \Omega} f d \lambda_\partial - \int_{\mathcal L} f d \lambda_\mathcal L \right)^2 
= \left( \int_{\partial \Omega} (f(x)-f(p_x)) d \lambda_\partial  \right)^2 
\leq \int_{\partial \Omega} (f(x)-f(p_x))^2 d \lambda_\partial . 
\end{align*}
Denoting by $\lambda_\partial^+$ and $\lambda_\partial^-$  the normalized Hausdorff measures on $\partial \Omega^+:=  \{(x,y) \in \partial \Omega:  y>0\}$ and $\partial \Omega^-:=  \{(x,y) \in \partial \Omega:  y<0\}$, respectively, 
\begin{align*}
 \int_{\partial \Omega} (f(x)-f(p_x))^2 d \lambda_\partial 
 = \frac{1}{2}   \int_{\partial \Omega^+} (f(x)-f(p_x))^2 d \lambda^+_\partial 
 +\frac{1}{2} \int_{\partial \Omega^-} (f(x)-f(p_x))^2 d \lambda^-_\partial .
\end{align*}
Moreover, for  
any $\mC^1$-function $g:[-\pi, L] \to \R$, 
\begin{align*}
\frac{1}{\pi} \int_0^\pi \left|g(-\theta)-g(L-\textstyle \frac{\theta L}{\pi})\right|^2 d\theta 
\leq \frac{\pi+L}{2} \int_{-\pi}^L |g'(t)|^2  dt ,
\end{align*}
so we deduce, identifying $\partial \Omega^+$ with $[-\pi,0]$ and $\mathcal L$ with $[0,L]$, that
\begin{align*}
\int_{\partial \Omega^+} (f(x)-f(p_x))^2 d \lambda^+_\partial 
\leq \frac{\pi+L}{2} \left( \pi \int_ {\partial \Omega^+} \|\bnabla f\|^2 d \lambda^+_\partial  + L \int_\mathcal L  \|\bnabla f\|^2 d \lambda_\mathcal L \right)
\end{align*}
and using symmetry to deal with $\partial \Omega^-$, we obtain
\begin{align*}
 \int_{\partial \Omega} (f(x)-f(p_x))^2 d \lambda_\partial 
 &\leq \frac{\pi+L}{4}\left( \pi \int_ {\partial \Omega^+} \|\bnabla f\|^2 d \lambda^+_\partial +\pi \int_ {\partial \Omega^-} \|\bnabla f\|^2 d \lambda^-_\partial  + 2L \int_\mathcal L  \|\bnabla f\|^2 d \lambda_\mathcal L \right) \\
  &\leq \frac{(\pi+L)(2\pi+L)	}{2} \int_\Sigma \|\bnabla f\|^2 d \lambda_\Sigma.
\end{align*}
Putting together the above inequalities, we get
\begin{align*}
\left( \int_\Omega f d \lambda_\Omega - \int_\Sigma f d \lambda_\Sigma \right)^2
&\leq \frac{3}{8} \int_\Omega \| \nabla f \|^2 d\lambda_\Omega
+ 2\frac{L^2}{(2\pi +L)^2} \frac{(\pi+L)(2\pi+L)	}{2\beta} \int_\Sigma \beta\|\bnabla f\|^2 d \lambda_\Sigma 
\end{align*}
which leads to inequality~\eqref{ineq K} with $K_1=\frac{3}{8}$ and $K_2=\frac{L^2(\pi+L)}{\beta(2\pi + L)}$. 
\end{proof}

\begin{proof}[Proof of Proposition~\ref{prop:ball_needle}]
Since $\Cint=\infty$, we immediately get from Proposition~\ref{prop:interp_method} that 
\begin{align*}
C_\alpha \leq
\max \left(
C_\Omega + (1-\alpha)K_1 ,
\alpha K_2 ,
C_\Sigma + \alpha K_2 \right)
= \max \left(
C_\Omega + (1-\alpha)K_1 ,
C_\Sigma + \alpha K_2 \right).
\end{align*}
Therefore,  
\begin{align}
\label{needle_with_beta}
    C_\alpha \leq \max \left( \frac{1}{\sigma_{\Omega}} + \frac{3}{8}(1-\alpha), \frac{1}{\beta\gamma_L} + \alpha \frac{L^2(\pi+L)}{\beta(2\pi + L)} \right),
\end{align}
where $\sigma_{\Omega} \approx 3.39$. 
\end{proof}

\begin{remark}
(i)  If $\beta$ is large enough, that is if the diffusion velocity  is larger on $\Sigma$ than on $\Omega$, then the first term in~\eqref{needle_with_beta} dominates. Precisely, if $\beta \geq \sigma_\Omega \left( \frac{1}{\gamma_L} +  \frac{L^2(\pi+L)}{2\pi + L} \right)$, then \eqref{needle_with_beta} rewrites for any $\alpha$
  \[C_\alpha \leq \frac{1}{\sigma_{\Omega}} + \frac{3}{8}(1-\alpha).
  \]
  Conversely, if $\beta \leq \frac{1}{\gamma_L} \left(\frac{1}{\sigma_{\Omega}} + \frac{3}{8}\right)^{-1}$, then  \eqref{needle_with_beta} rewrites for any $\alpha$
  \[C_\alpha \leq \frac{1}{\beta\gamma_L} + \alpha \frac{L^2(\pi+L)}{\beta(2\pi + L)}.
  \]
(ii) Combining with    \cite{bormann2025cheegertypeinequalitydriftlaplacian} one might expect that weighted stratified spaces as in \cite{Ullrich:2025} can be treated similarly.
\end{remark}

\section{Continuity of \texorpdfstring{$C_\alpha$}{C alpha}}
\label{subsection:cont_section} 
In general, the  function $\alpha\mapsto C_\alpha$ 
might have discontinuities at  $\alpha\in \{0,1\}$ in which cases an upper bound for $C_\alpha$ which interpolates continuously between $C_0$ and $C_1$ cannot exist due to the lower semicontinuity of $\alpha\mapsto C_\alpha$. For example, when $\Omega=(0,b) \times (0,1)\subset \R^2$ and $\Sigma=[0,b]\times\{0\}$, straightforward  computations yield \[
  \lim_{ \alpha \to  0 }C_{\alpha}= \max\left\{ C_\Sigma, \frac{4}{ \pi^2 } \right\},
\] 
where $C_\Sigma= \frac{b^2}{\pi^2}$. Hence $\alpha \mapsto C_{\alpha}$  is discontinuous at $\alpha=0$   if and only if $b< 2$.  
-- To generalize this to the framework of Section~\ref{subsec:energy_decomposition}  let 
   $\mC^1_0(\overline{\Omega})=\{f \in \mC^1(\overline{\Omega}) : f=0 \mbox{ on }\Sigma\}$  
and 
  \[
    \tilde{C}_0:=\sup_{\substack {f \in \mC^1_0(\overline \Omega) \\ f \text{ non constant}}}  \frac{\int_\Omega f^2d\lambda_\Omega}{\int_\Omega \| \Deriv f\|^2d\lambda_\Omega}.
  \]
  (If $\Deriv = \nabla$, $\tilde C_0$ is the inverse of the spectral gap for Brownian motion on $\Omega$ with killing on $\Sigma$ and normal reflection at $\partial \Omega \setminus \Sigma$.)
We can then record the following statement as a partial corollary to Proposition~\ref{prop:interp_method}.

\begin{proposition} 
\label{prop:continuity_condition}
In the setting of proposition \ref{prop:interp_method} it holds that  \[
    \varliminf_{ \alpha \to 0 }C_{\alpha}\geq \tilde{C}_0.
  \]
In particular,  if $C_\Sigma < \tilde C_0$, then $\alpha \mapsto C_\alpha$ is discontinuous at $\alpha=0$. Conversely, if $C_\Sigma \geq C_\Omega+K_1$ then $\alpha \mapsto C_\alpha$ is continuous at 0. If $C_{\Omega} \geq K_2$ continuity at 1 holds.
\end{proposition}

\begin{proof} 
To prove the first statement, take a non constant function $g \in \mC^1_{0}(\overline\Omega)$ and estimate 
  \begin{align*}
    \varliminf_{ \alpha \to 0 }C_{\alpha}&= \varliminf_{ \alpha \to 0 }\sup_{\substack {f \in \mC^1(\overline \Omega) \\ f \text{ non constant}}}  \frac{\Var_{\lambda_\alpha} f}{\e_\alpha (f)}\geq \varliminf_{ \alpha \to 0 }\frac{\Var_{\lambda_\alpha} g}{\e_\alpha (g)}\\
    &= \varliminf_{ \alpha \to 0 } \frac{ \alpha\Var_{\lambda_{\Omega}}g+(1-\alpha)\Var_{\lambda_{\Sigma}}g+\alpha(1-\alpha)\left( \int_{ \Omega }   gd \lambda_{\Omega}-\int_{ \Sigma }   gd \lambda_{\Sigma}   \right)^2 }{ \alpha \int_{ \Omega }   \|\Deriv g\|^2 d\lambda_{\Omega}+(1-\alpha)\int_{ \Sigma}   \|\bDeriv g\|^2 d \lambda_{\Sigma }   }.
  \end{align*}
 Since $g=0$ on $\Sigma$, we obtain
  \begin{align*}
    \varliminf_{ \alpha \to 0 }C_{\alpha}&\geq \varliminf_{ \alpha \to 0 } \frac{ \alpha\Var_{\lambda_{\Omega}}g+\alpha(1-\alpha)\left( \int_{ \Omega }   gd \lambda_{\Omega}  \right)^2 }{ \alpha \int_{ \Omega }   \|\Deriv g\|^2 d \lambda_{\Omega}  } = \frac{ \int_\Omega g^2 d\lambda_\Omega }{ \int_{ \Omega }   \|\Deriv g\|^2 d \lambda_{\Omega}  } .
  \end{align*}
Taking the supremum over $g \in \mC^1_0(\overline\Omega)$  yields the first statement. 
  
To prove the second assertion  note that $ \alpha \mapsto C_\alpha$ is the pointwise supremum of a family of continuous functions and therefore lower semi continuous. 
Thus $C_\Sigma=C_0 \leq \varliminf_{\alpha \to 0} C_\alpha$. 
If $C_\Sigma \geq C_\Omega+K_1$, the r.h.s. of inequality~\eqref{inequality_interpolation} converges to $C_\Sigma$ as $\alpha$ goes to 0, which implies that  $ \varlimsup_{\alpha \to 0} C_\alpha  \leq C_\Sigma$. Similarly, if $C_{\Omega} \geq K_2$, the r.h.s. of~\eqref{inequality_interpolation} converges to $C_\Omega$ as $\alpha$ goes to 1.
\end{proof}

\begin{remark}  For smooth enough boundary 
the constant $K_2$ can always be taken equal to zero, hence by proposition~\ref{prop:continuity_condition} continuity at $\alpha=1$ holds. An example where a phase transition appears at $\alpha=0$ is given in section~\ref{subsec:partial_ball}.  In the example of section~\ref{subsec:needle} it holds that $C_\Omega<K_2$ but continuity of at $\alpha =1$ can  be established via Mosco-convergence \cite{MR1283033} of the associated Dirichlet forms, see also \cite{MR3154581}.
\end{remark}

\providecommand{\bysame}{\leavevmode\hbox to3em{\hrulefill}\thinspace}
\providecommand{\MR}{\relax\ifhmode\unskip\space\fi MR }
\providecommand{\MRhref}[2]{%
	\href{http://www.ams.org/mathscinet-getitem?mr=#1}{#2}
}
\providecommand{\href}[2]{#2}

\end{document}